\documentclass[a4paper,10pt, reqno]{amsart}

\usepackage{hyperref}
\hypersetup{
    colorlinks=true,       % false: boxed links; true: colored links
    linkcolor=blue,          % color of internal links
    citecolor=blue,        % color of links to bibliography
    filecolor=blue,      % color of file links
    urlcolor=blue           % color of external links
}
\numberwithin{equation}{section}
\usepackage[all]{xy}
\usepackage{tikz,graphicx}
\usepackage{amssymb}
\usepackage{booktabs}
\usepackage[english]{babel}
\usepackage{multirow}
\usepackage{multicol}
\DeclareMathOperator{\modu}{{{mod}}}
\theoremstyle{plain}
\newtheorem{thm}{Theorem}[section]
\newtheorem{lem}[thm]{Lemma}
\newtheorem{prop}[thm]{Proposition}
\newtheorem{cor}[thm]{Corollary}
\newtheorem{rem}[thm]{Remark}

\theoremstyle{definition}
\newtheorem{defn}[thm]{Definition}
\newtheorem{example}[thm]{Example}

\newcommand{\s}{\vspace{0.3cm}}

\newcommand{\Z}{\mathbb{Z}}

\usepackage{enumerate}

\begin{document}
\title[A class of pseudoreal Riemann surfaces]{A class of pseudoreal Riemann surfaces with diagonal automorphism group}

\author[Eslam Badr] {Eslam Badr}
\address{$\bullet$\,\,Eslam Badr}
%\address{Departament Matem\`atiques, Edif. C, Universitat Aut\`onoma de Barcelona\\
%08193 Bellaterra, Catalonia, Spain} \email{eslam@mat.uab.cat}
\address{Department of Mathematics,
Faculty of Science, Cairo University, 12613 Giza-Egypt}
\email{eslam@sci.cu.edu.eg}

%\thanks{E. Badr is supported by MTM2016-75980-P}


\maketitle
\begin{abstract}
A Riemann surface $\mathcal{S}$ having field of moduli $\mathbb{R}$, but not a field of definition, is called \emph{pseudoreal}. This means that $\mathcal{S}$ has anticonformal automorphisms, but non of them is an involution. We call a Riemann surface $\mathcal{S}$ \emph{plane} if it can be described by a smooth plane model of some degree $d\geq4$ in $\mathbb{P}^2_{\mathbb{C}}$.

%We know from B. Huggins \cite[Chp. 6]{Hug} that a Riemann surface $\mathcal{S}$ is pseudoreal-plane only if its conformal automorphism group $\operatorname{Aut}_+(\mathcal{S})$ is $\operatorname{PGL}_3(\mathbb{C})$-conjugate to the Hessian group $\operatorname{Hess}_{36}$ of order $36$ or to a diagonal group $G$ (i.e. a finite non-trivial group made entirely of $3\times3$ projective linear matrices of diagonal shapes).
%
We characterize pseudoreal-plane Riemann surfaces $\mathcal{S}$, whose conformal automorphism group $\operatorname{Aut}_+(\mathcal{S})$ is $\operatorname{PGL}_3(\mathbb{C})$-conjugate to a finite non-trivial group $G$ that leaves invariant infinitely many points of $\mathbb{P}^2_{\mathbb{C}}$. In particular, we show that such pseudoreal-plane Riemann surfaces exist only if $\operatorname{Aut}_+(\mathcal{S})$ is cyclic of even order $n$ dividing the degree $d$. Explicit examples are given, for any  degree $d=2pm$ with $m>1$ odd, $p$ is prime and $n=d/p$.
\end{abstract}

\section{Introduction}
Any Riemann surface $\mathcal{S}$ of genus $g$ can be understood by a certain irreducible projective curve $\mathcal{C}$ of genus $g$ over $\mathbb{C}$. A complex subfield $K$ is \emph{a field of definition for $\mathcal{S}$} if $\mathcal{C}$ can be defined over $K$. \emph{The field of moduli of $\mathcal{S}$} is the intersection of all fields
of definition for $\mathcal{S}$.

There exists in the literature an alternative definition for the field of moduli, relative to a given field extension $L/K$, which is commonly used; Given an irreducible projective curve $\mathcal{C}$ defined over a field $L$, the field of moduli of $\mathcal{C}$ relative to $L/K$, denoted by $M_{L/K}(\mathcal{C})$, is the fixed subfield of $L$ by the subgroup
$$U_{L/K}(\mathcal{C}):=\{\sigma\in\operatorname{Gal}(L/K)\,|\,\mathcal{C}\,\,\text{isomorphic over}\,\,L\,\,\text{to}\,\,^{\sigma}\mathcal{C}\},$$
where $\operatorname{Gal}(L/K)$ denotes the group of all $K$-automorphisms of the field $L$. We are interested in the particular case when $K=\mathbb{R}$ and $L=\mathbb{C}$.

A necessary and sufficient condition for the field of moduli to be a field of definition was provided by Weil \cite{We}:
\begin{thm}[Weil's criterion of descent]\label{weilcocylcle}
Let $\mathcal{C}$ be an irreducible projective algebraic curve defined over a field $L$, and let $L/K$ be a finite Galois extension. There is an irreducible projective algebraic curve $\mathcal{C}'$ over $K$ and an $L$-isomorphism $g:\mathcal{C}'\times_KL\rightarrow\mathcal{C}$ if and only if, for any $\sigma\in\operatorname{Gal}(L/K)$, there exists an $L$-isomorphism $f_{\sigma}:\,^{\sigma}\mathcal{C}\mapsto\,\mathcal{C}$ such that $f_{\sigma}\circ\,^{\sigma}f_{\tau}=f_{\sigma\tau}\,\,\text{for all}\,\,\sigma,\tau\in\operatorname{Gal}(L/K).$
Moreover, $f_{\sigma}\circ\,^{\sigma}g=g$ for all $\sigma\in\operatorname{Gal}(L/K)$.
\end{thm}
\begin{rem}
A. Weil in \cite{We} observed that the
above result is valid as long the overfield is finitely generated over the
lower field. In particular, it holds for $L=\mathbb{C}$, $K=\mathbb{Q}$ and $\mathcal{C}$ of genus $g\geq2$ (cf. \cite[Remark 2.2]{BaHidQu}).
\end{rem}
A Riemann surfaces $\mathcal{S}$ with trivial automorphism group needs to be defined over its field of moduli, since the above condition becomes trivially true. It is also known that the field of moduli is a field of definition for $\mathcal{S}$ when the genus $g$ is at most $1$. However, if $g>1$ and $\operatorname{Aut}_+(\mathcal{S})$ is non-trivial, then Weil's conditions are difficult to be checked and so there is no guarantee that the field of moduli is a field of definition.
This was first pointed out by Earle \cite{earle} and Shimura \cite{shimura}. More concretely, in page 177 of \cite{shimura}, the first examples not definable over their field of moduli are introduced, which are hyperelliptic curves over $\mathbb{C}$ with two automorphisms. There are also examples of non-hyperelliptic curves not definable over their field of moduli in \cite{SAUL1, hidalgo1, Hug, Kont}. For all these (explicit) examples, the field of moduli is always a subfield of $\mathbb{R}$, and they can be defined over an imaginary extension of degree two of the field of moduli. The present author and et. al. constructed in \cite{BaHidQu} the first (explicit) examples of Riemann surfaces, which are definable over the
reals but cannot be defined over the field of moduli.

\subsection{Case of hyperelliptic curves}
By the work of Mestre \cite{Me}, Huggins \cite{Hug2, Hug},
Lercier-Ritzenthaler \cite{LeRi} and Lercier-Ritzenthaler-Sijsling
\cite{LRS}, one gets the answer in the case of hyperelliptic curves $\mathcal{C}$ over a perfect field $L$ of characteristic $p\neq2$. We issue the next table from \cite{LRS}:
\begin{center}
\begin{tabular}{|c|c|c|c|}
\hline
$H=\operatorname{Aut}_+(\mathcal{C})/\langle\iota\rangle$ & Conditions & The field of moduli=\\
                                                         &            & A field of definition\\\hline
Not tamely cyclic                                        &            & Yes\\\hline
\multirow{2}{*}{Tamely cyclic with $\# H>1$}             &$g$ odd, $\#H$odd  &    Yes\\
%\hline
&$g$ even or $\#H$ even &No\\
\hline \multirow{2}{*}{Tamely cyclic with $\#H=1$}&$g$ odd &Yes\\
&$g$ even& No\\
\hline
\end{tabular}
\end{center}
\s
By \emph{tamely cyclic}, we mean that the group is cyclic of order
not divisible by the characteristic $p$.
%%%%%%%%%%%%%%%%%%%%%%%%%%%%%%%%%%%%%%%%%%%%%%%%%%%%%%%%
\subsection{Case of smooth plane curves}
A linear transformation $A=(a_{i,j})$ of the projective plane $\mathbb{P}^2_{L}$ over a field $L$ is often written as $$[a_{1,1}X+a_{1,2}Y+a_{1,3}Z:a_{2,1}X+a_{2,2}Y+a_{2,3}Z:
a_{3,1}X+a_{3,2}Y+a_{3,3}Z],$$
where $\{X,Y,Z\}$ are the homogenous coordinates of $\mathbb{P}^2_L$.

\begin{defn}
A smooth projective curve $\mathcal{C}$ of genus $g\geq3$ over an algebraically closed field $\overline{L}$ is called a \emph{smooth plane curve of genus $g$ over $\overline{L}$} if $\mathcal{C}$ is $\overline{L}$-isomorphic to a non-singular plane model $F_{\mathcal{C}}(X,Y,Z)=0$ in $\mathbb{P}^2_{\overline{L}}$, where $F_{\mathcal{C}}(X,Y,Z)$ is a homogenous polynomial of degree $d$ with coefficients in $\overline{L}$, and $g=(d-1)(d-2)/2$.
\end{defn}
The problem for smooth plane curves was addresses by B. Huggins in \cite[Chps. $6$ and $7$]{Hug}. In particular, we have:
\begin{thm}(B. Huggins, \cite[Theorem 6.4.8]{Hug})\label{hugginsconc}
Let $\overline{L}$ be a fixed algebraic closure of a perfect field $L$ of characteristic $p\neq2$. Then, if $\mathcal{C}$ is a smooth plane curve of genus $g\geq3$ defined over $\overline{L}$, then the field of moduli $M_{\overline{L}/L}(\mathcal{C})$ for $\mathcal{C}$, relative to the Galois extension $\overline{L}/L$, % it seems that $\overline{K}/K$ is Galois iff K is perfect see Christian Master thesis
is a field of definition if the automorphism group $\operatorname{Aut}_+(\mathcal{C})$ is not $\operatorname{PGL}_3(\overline{L})$-conjugate to a diagonal subgroup of $\operatorname{PGL}_3(\overline{L})$, or to one of the Hessian groups $\operatorname{Hess}_*$ with $*\in\{18,36\}$, or a semidirect product $\mathcal{B}\rtimes\mathcal{A}$ for some finite diagonal subgroup $\mathcal{A}$ of $\operatorname{PGL}_3(L)$ and a non-trivial $p$-group $\mathcal{B}$ consisting entirely of elements of the shape $[X:\alpha X+Y:\beta X+\gamma Y+Z]$.
\end{thm}
%%%%%%%%%%%%%%%%%%%%%%%%%%%%%%%%%%%%%%%%%%%%%%%%%
\subsection{Pseudoreal Riemann surfaces}
The surface $\mathcal{S}$ is called \emph{real} if ${\mathbb R}$ is a field of definition of it; this is equivalent for $\mathcal{S}$ to admit an anticonformal automorphism of order two (as a consequence of Weil's descent theorem \cite{We}). Also, the field of moduli of $\mathcal{S}$ is a subfield of ${\mathbb R}$ if and only if it is isomorphic to its complex conjugate, equivalently, if it admits anticonformal automorphisms \cite{earle, shimura, Silhol}. Those surfaces of genus $g$ with real field of moduli corresponds to the real points of the moduli space ${\mathcal M}_{g}$. Riemann surfaces whose field of moduli is real but are not real are usually called {\it pseudo-real}.

Theorem $1$ in \cite[p. 48]{Sing} assures the existence of pseudoreal Riemann surfaces $\mathcal{S}$ for any genus $g\geq2$. However, it does not
construct algebraic models for $\mathcal{S}$. Moreover, once the genus $g\geq2$ is fixed and the list of all automorphism groups acting on Riemann surfaces of this genus $g$, with their signature and generating vectors, is determined,
then one can classify the automorphism groups
of pseudoreal Riemann surfaces of genus $g$ by using the algorithm recently exposed in \cite{AQC}.
\subsubsection{Pseudoreal-plane Riemann surfaces}
A pseudoreal Riemann surface $\mathcal{S}$ of genus $g\geq3$ admitting a non-singular plane model in $\mathbb{P}^2_{\mathbb{C}}$ is called \emph{pseudoreal-plane}. In this case, $\mathcal{S}$ is non-hyperlliptic of genus $g=(d-1)(d-2)/2$ for some integer $d\geq4$ (cf. \cite[Lemma 4.1]{BaHidQu}).

It has been announced (without a proof) in \cite[p. 136]{Hug} that pseudoreal-plane Riemann surfaces with automorphism group $\operatorname{Hess}_{18}$ do not exist. Furthermore, we have seen explicit examples of pseudoreal-plane Riemann surfaces whose conformal automorphism group is $\operatorname{Hess}_{36}$ or diagonal in \cite[Chp. 7]{Hug}. Therefore, by the virtue of Theorem \ref{hugginsconc}, one concludes:
\begin{cor}
A Riemann surface $\mathcal{S}$ is pseudoreal-plane only if $\operatorname{Aut}_+(\mathcal{S})$ is $\operatorname{PGL}_3(\mathbb{C})$-conjugate to $\operatorname{Hess}_{36}$ or to a diagonal group.
\end{cor}
Because of the above results, we were wondering, once the genus $g\geq3$ is fixed, about the existence of pseudoreal-plane Riemann surfaces $\mathcal{S}$ of genus $g$ whose $\operatorname{Aut}_+(\mathcal{S})$ is diagonal. We will see that, in contrast to pseudreal Riemann surfaces, that the answer in general is \textbf{No}, that is they are not always exist. We also characterize $\operatorname{Aut}_+(\mathcal{S})$ for pseudoreal-plane Riemann surfaces and give explicit examples.
\subsection*{Acknowledgments}
I feel fortunate that I had the opportunity of learning from Francesc Bars and doing my PhD under his supervision at UAB, Spain. I also would like to thank Rub\'en A. Hidalgo and Sa\'ul Quispe for their useful comments.
%\section{The field of moduli and fields of definition}\label{basicdefns}

%\begin{prop}(D\'{e}bes-Emsalem, \cite[Proposition 2.1]{DeEm1})\label{theassumptionnotrestrict}
%Let $C$ be a smooth curve over $F$ and let $F/K$ be a Galois extension. The group $U_{F/K}(C)$ is a closed subgroup of $\operatorname{Gal}(F/K)$ with respect to the Krull topology. That is,
%$$U_{F/K}(C)=\operatorname{Gal}(F/M_{F/K}(C)).$$
%The field of moduli $M_{F/K}(C)$ of $C$ relative to the extension $F/K$ is contained in each field of definition of $C$ between $K$ and $F$ (in particular, it is a finite extension of $K$). % it seems by the proof and that we are in Galois extension that \operatorname{Gal}(F/M_{F/K}(C)) is a normal subgroup of Gal(F/K) of ifnite index!!!
%Hence if the field of moduli is a field of definition, it
%is the smallest field of definition between $K$ and $F$. Finally, if $L:=M_{F/K}(C)$, then the field of moduli of $C$
%relative to the extension $F/L$ is exactly $L$.
%\end{prop}
%\begin{rem}\label{theassumptionnotrestrict2}
%The final observation of Proposition \ref{theassumptionnotrestrict} that the field of moduli relative to
%the extension $F/M_{F/K}(C)$ equals $M_{F/K}(C)$ generally allows one to reduce to the
%situation where the base field $K$ is the field of moduli of the given curve $C$, by extending the scalars from $K$ to $M_{F/K}(C)$.
%\end{rem}
%%%%%%%%%%%%%%%%%%%%%%%%%%%%%%%%%%%%%%%%%%%%%%%%%%%%%%%%%%%%%%%%55
\section{On conformal automorphism groups of plane Riemann surfaces}
%As a consequence of Riemann-H\"{u}rwitz formula and B\'{e}zout theorem, we have:
%\begin{lem}\label{nonhypergenus}
%A plane Riemann surface $\mathcal{S}$ of degree $d\geq4$ is always non-hyperelliptic of genus $g=(d-1)(d-2)/2$.
%\end{lem}

\begin{defn}
For a non-zero monomial $cX^iY^jZ^k$
with $c\in\mathbb{C}\setminus\{0\}$, its exponent is defined to be $max\{i,j,k\}$. For a homogenous polynomial $F(X,Y,Z)$, the core of it is
defined to be the sum of all terms of $F$ with the greatest
exponent. Now, let $\mathcal{S}_0$ be a plane Riemann surfaces, a pair $(\mathcal{S},H)$ with
$H\leq \operatorname{Aut}_+(\mathcal{S})$ is said to be a descendant of $\mathcal{S}_0$ if $\mathcal{S}$ is defined
by a homogenous polynomial whose core is a defining polynomial of
$\mathcal{S}_0$ and $H$ acts on $\mathcal{S}_0$ under a suitable change of the
coordinates system, i.e. $H$ is conjugate to a subgroup of
$\operatorname{Aut}_+(\mathcal{S}_0)$.
\end{defn}
By \cite[\S 1-10]{Mit} and the proof of Theorem 2.1 in \cite{Ha}, we conclude:
\begin{thm}[Mitchell \cite{{Mit}}, Harui \cite{Ha}]\label{Harui,Mitchell}
	Let $G$ be a subgroup of conformal automorphisms of a plane Riemann surface $\mathcal{S}$ of degree $d\geq4$. Then, one of the following holds:
	\begin{enumerate}[(i)]
		\item $G$ fixes a line in $\mathbb{P}^2_{\mathbb{C}}$ and a point off this line.
		\item $G$ fixes a triangle $\Delta\subset\mathbb{P}^2_{\mathbb{C}}$, i.e. a set of three non-concurrent lines, and neither line nor a point is leaved invariant. In this case, $(\mathcal{S},G)$ is a descendant of the the Fermat curve $F_d:\,X^d+Y^d+Z^d=0$ or the Klein curve $K_d:\,XY^{d-1}+YZ^{d-1}+ZX^{d-1}=0$.
		\item $G$ is $\operatorname{PGL}_3(\mathbb{C})$-conjugate to a finite primitive subgroup namely, the Klein group
		$\operatorname{PSL}(2,7)$, the icosahedral group $\operatorname{A}_5$, the alternating group
		$\operatorname{A}_6$, the Hessian group $\operatorname{Hess}_{*}$ with $*\in\{36,72,216\}$.
		\end{enumerate}
\end{thm}
We use $\zeta_n$ for a fixed primitive $n$-th root of unity
inside $\mathbb{C}$.
\begin{defn}\label{homologydefn}
By an \emph{homology} of period $n\in\mathbb{Z}_{\geq1}$, we mean a projective linear transformation of the plane $\mathbb{P}^2_{\mathbb{C}}$, which acts up to $\operatorname{PGL}_3(\mathbb{C})$-conjugation, as $$(X:Y:Z)\mapsto (\zeta_{n}X:Y:Z).$$
Such a transformation fixes pointwise a line (its axis) and a point off this line (its center). Otherwise, it is called a \emph{non-homology}
\end{defn}
\begin{thm}[Mitchell \cite{Mit}]\label{homologies}
Let $G$ be a finite group of $\operatorname{PGL}_3(\mathbb{C})$. If $G$ contains an homology of period $n\geq4$, then it fixes a point, a line or a triangle. Moreover, the Hessian group $\operatorname{Hess}_{216}$ is the only finite subgroup of $\operatorname{PGL}_3(\mathbb{C})$ that contains homologies of period $n=3$, and does not leave invariant a point, a line or a triangle.
\end{thm}

%%%%%%%%%%%%%%%%%%%%%%%%%%%%%%%%%%%%%%%%%%%%%%%%%%%%%%%%%%%%%%%%%
\subsection{Preliminaries on diagonal conformal automorphism groups}\label{structurediagonalgroups}
The group of all $3\times3$ projective linear matrices of diagonal shapes is denoted by $\operatorname{D}(\mathbb{C})$. A finite non-trivial group $G$, representable inside $\operatorname{PGL}_3(\mathbb{C})$, is said to be \emph{diagonal} if $\varrho(G)$ is $\operatorname{PGL}_3(\mathbb{C})$-conjugate to a subgroup of $\operatorname{D}(\mathbb{C})$, for some injective representation $\varrho:G\hookrightarrow\operatorname{PGL}_3(\mathbb{C})$.

Firstly, we show:
\begin{prop}\label{classific.}
Let $G$ be a finite non-trivial diagonal group, consisting entirely of homologies. Then, it is either cyclic or conjugate to  $$\varrho_0(\Z/2\Z\times\Z/2\Z):=\langle\operatorname{diag}(1,-1,1),\operatorname{diag}(1,1,-1)\rangle.$$
\end{prop}
\begin{proof}
Fix an injective representation $\varrho(G)\leq\operatorname{D}(\mathbb{C})$ and suppose that  $\varrho(G)\neq\varrho_0(\Z/2\Z\times\Z/2\Z)$. Hence, there must be an homology $\phi\in\varrho(G)$ of order $m>2$, since $\varrho_0(\Z/2\Z\times\Z/2\Z)$ is the unique non-cyclic diagonal subgroup whose elements have orders at most $2$. % any homology of order 2 in its diagonal form lies in \langle diag(1,-1,1),\,diag(1,1,-1)\rangle\simeq \Z_2\times \Z_2
There is no loss of generality to assume that $\phi=\operatorname{diag}(1,1,\zeta_m)$, in particular its axis is the reference line $L_3:Z=0$ and its center is the reference point $P_1=(1:0:0)$. Take $\psi\in\varrho(G)\setminus\langle\phi\rangle$ (if $\varrho(G)\setminus\langle\phi\rangle=\emptyset$, then $\varrho(G)$ is cyclic and there is nothing to prove further). If $\psi$ has a different axis from $L_3$, then $\phi^s\psi\in\varrho(G)$ is a non-homology for a suitable choice of the integer $s$ because $m>2$; for example, write $\psi$ as $\operatorname{diag}(1,\zeta_{m'},1)$ for some integer $m'>1$, hence $\phi^s\psi=\operatorname{diag}(1,\zeta_{m'},\zeta^s_m)$. So, we may set $s=1$ when $m\neq m'$, and $s=2$ otherwise. In both cases, $\phi^s\psi$ is a non-homology in $\varrho(G)$, which conflicts our assumption that $\varrho(G)$ is made entirely of homologies (Definition \ref{homologydefn}). Therefore, all elements of $\varrho(G)$ admit the same axis and the same center, i.e. each is of the shape $\operatorname{diag}(1,1,\zeta_n)$ for some $n\in\mathbb{N}$. Consequently, $\varrho(G)$ is contained in the cyclic group generated by $\operatorname{diag}(1,1,\zeta_{n_0})$, where $n_0$ is the least common multiple of the orders of the elements of $\varrho(G)$. Thus $\varrho(G)$ is itself cyclic.
\end{proof}
\begin{defn}[The Hessian groups]\label{hessiandefn}
By $\operatorname{Hess}_{18}$ we mean the group of order $18$ generated by $S:=[X:\zeta_3Y:\zeta_3^2],T:=[Y:Z:X]$ and $R:=[X:Z:Y]$. The group $\operatorname{Hess}_{36}$ is the one generated by $\operatorname{Hess}_{18}$ and $V:=[X+Y+Z:X+\zeta_3Y+\zeta_3^2Z:X+\zeta_3^2Y+\zeta_3Z]$.
\end{defn}
\begin{prop}\label{hessiansnotdiagonal}
The Hessian groups $\operatorname{Hess}_*$, for $*\in\{18,36\}$, are not diagonal.
\end{prop}
\begin{proof}
Assume on the contrary that $\varrho(\operatorname{Hess}_*)\leq\operatorname{D}(\mathbb{C})$ for some $\varrho$.
By Definition \ref{hessiandefn}, $\varrho(\operatorname{Hess}_*)$ contains a non-homology $\phi$ of order $3$, and we may write it as $\operatorname{diag}(1,\zeta_3,\zeta_3^2)$.
Moreover, there should be another element $\psi$ of order $3$, such that $\psi\phi=\phi\psi$. Hence $\psi$ should be of the shape $\{[Y:Z:X],\,[Z:X:Y]\}$ modulo $\operatorname{D}(\mathbb{C})$. Thus $\psi\notin\operatorname{D}(\mathbb{C})$, which conflicts the fact $\varrho(\operatorname{Hess}_*)\leq\operatorname{D}(\mathbb{C})$.
\end{proof}

\begin{lem}[Mitchell \cite{Mit}]\label{fixedpoints}
Let $\phi$ be a $3\times3$ complex projective linear transformation of order $n\in\mathbb{Z}_{\geq1}$.
\begin{enumerate}[(i)]
  \item If $\phi$ is an homology, then the fixed points of $\phi$ consists entirely of its center and all points on its axis. In particular, every triangle whose set of vertices is pointwise fixed by $\phi$ contains its center as a vertex.
  \item If $\phi$ is a non-homology, then it fixes exactly three points. In particular, there is a unique triangle whose vertices are pointwise fixed by $\phi$.
\end{enumerate}
\end{lem}

Let $\varrho(G)$ be a finite non-trivial group in $\operatorname{PGL}_3(\mathbb{C})$. If $\mathcal{F}_{\varrho,G}(X,Y,Z)=0$ is a family of plane Riemann surfaces in $\mathbb{P}^2_{\mathbb{C}}$, with conformal automorphism group exactly $\varrho(G)$, then isomorphisms between two curves in the same family $\mathcal{F}_{\varrho,G}(X,Y,Z)=0$ (in particular with identical conformal automorphism group) are clearly given by $3\times3$ projective matrices in the normalizer $N_{\varrho(G)}(\mathbb{C})$ of $\varrho(G)$ in $\operatorname{PGL}_3(\mathbb{C})$. Therefore, it is practical to compute the normalizer for each case.
\begin{prop}[Normalizer]\label{normalizerdetyermination}
Let $\varrho(G)\leq\operatorname{D}(\mathbb{C})$ be a finite non-trivial group.
\begin{enumerate}[(i)]
  \item If $\varrho(G)$ contains a non-homology $\phi=\operatorname{diag}(\zeta_n^a,\zeta_n^b,1)$, then $N_{\varrho(G)}(\mathbb{C})=\langle\operatorname{D}(\mathbb{C}), H\rangle$ for some $H\leq\tilde{\operatorname{S}}_3$, where $\tilde{\operatorname{S}}_3$ is the symmetry group $\langle[X:Z:Y],[Z:X:Y]\rangle$ of order $6$.
  \item If $\varrho(G)$ is generated by an homology $\phi=\operatorname{diag}(1,1,\zeta_n)$ for some $n\in\mathbb{Z}_{\geq2}$, then  $N_{\varrho(G)}(\mathbb{C})=\operatorname{GL}_{2,Z}(\mathbb{C})$, where $\operatorname{GL}_{2,Z}(\mathbb{C})$ is the group of all projective linear matrices of the shape
$$\left(
\begin{array}{ccc}
\ast & \ast & 0 \\
\ast& \ast & 0 \\
0 & 0 & 1 \\
\end{array}
\right).$$
\end{enumerate}
\end{prop}
\begin{proof}
Using Lemma \ref{fixedpoints} and the assumption that $\varrho(G)$ is diagonal, we deduce that there is always a unique set $V$, which is fixed pointwise by $\varrho(G)$. We get $V=\{P_1=(1:0:0),\,P_2=(0:1:0),\,P_3=(0:0:1)\}$ if a non-homology is present in $\varrho(G)$, and $V$ is formed by all points of the line $L_3:Z=0$ and the point $P_3$, otherwise. Therefore, $V$ is also fixed by $N_{\varrho(G)}(\mathbb{C})$, and the computations are straightforward.
\end{proof}

%%%%%%%%%%%%%%%%%%%%%%%%%%%%%%%%%%%%%%%%%%%%%%%%%%%%%%%%%%%%%%%%%%%%%%%%%%%

\section{Pseudoreal-plane Riemann surfaces with diagonal automorphism group, fixing infinitely many points in $\mathbb{P}^2_{\mathbb{C}}$}\label{thecycliccase}
Let $\mathcal{M}_g$ be the (coarse) moduli space of smooth curves of genus $g\geq3$ over $\mathbb{C}$.
\begin{defn}\label{pseudoreal-planefirsttype}
By $\widetilde{{(\mathcal{M}_{g}^{Pl})}}_{n,\operatorname{diag}}^{h}$ where $n\geq2$ is an integer, we mean the substratum of $\mathcal{M}_{g}$ of plane Riemann surfaces $\mathcal{S}$ of genus $g=(d-1)(d-2)/2\geq3$, such that $\operatorname{Aut}_+(\mathcal{S})$ is a finite diagonal subgroup of $\operatorname{PGL}_3(\mathbb{C})$ made entirely of homologies. We also can assume that $n$ is maximal, in the sense that any automorphism of $\mathcal{S}$ has order at most $n$.
\end{defn}
Given $\mathcal{S}\in\widetilde{{(\mathcal{M}_{g}^{Pl})}}_{n,\operatorname{diag}}^{h}$, we know by Proposition \ref{classific.} that $\operatorname{Aut}_+(\mathcal{S})$ is either cyclic or $\operatorname{PGL}_3(\mathbb{C})$-conjugate to $\langle\operatorname{diag}(1,-1,1),\operatorname{diag}(1,1,-1)\rangle$. Fix a non-singular plane model $F_{\mathcal{S}}(X,Y,Z)=0$ for $\mathcal{S}$ over $\mathbb{C}$ of degree $d$, such that $\operatorname{Aut}_+(F_{\mathcal{S}})\leq\operatorname{D}(\mathbb{C})$. Badr-Bars in \cite{BaBa2}, following the techniques of Dolgachev in \cite{Dol} for genus $3$ curves (see also \cite[\S 2.1]{Bars}), determined the possible cyclic groups, which can appear inside $\operatorname{Aut}_+(F_{\mathcal{S}})$, moreover the associated defining equation is also given in each situation. In particular,
one deduces that the  stratum $\widetilde{{(\mathcal{M}_{g}^{Pl})}}_{n,\operatorname{diag}}^{h}$ might be non-empty only if $n$ divides $d$ or $d-1$.

For $u=1,2$, consider the sets $$S(u)_n:=\{u\leq j\leq d-1:\,\,d-j\equiv0\,(\modu\,n)\}.$$
\begin{prop}(Badr-Bars \cite[Theorem 7]{BaBa2})\label{geomcomp}
The stratum $\widetilde{{(\mathcal{M}_{g}^{Pl})}}_{n,\operatorname{diag}}^{h}$ is non-empty only if $n$ divides $d$ or $d-1$. More concretely, if $n\,|\,d$, then $F_{\mathcal{S}}(X,Y,Z)=0$ has the form
\begin{eqnarray*}
\mathcal{C}_1:\,Z^d+\sum_{j\in S(1)_n}Z^{d-j}L_{j,Z}+L_{d,Z}=0,
\end{eqnarray*}
% remark that $j\neq1$ by assumption that $n|d$. Also $j\neq d-1$ by non-triviality of the automorphism group
and if $n\,|\,d-1$, we get the form
\begin{eqnarray*}
\mathcal{C}_2:\,Z^{d-1}Y+\sum_{j\in S(2)_n}Z^{d-j}L_{j,Z}+L_{d,Z}=0.
\end{eqnarray*}
% remark that $j\neq2$ by assumption that $n|d-1$. Also $j\neq d-1$ by non-triviality of the automorphism group
Here $L_{j,Z}$ is a generic homogenous polynomial of degree $j$ where the variable $Z$ does not appear.
\end{prop}
%\begin{proof}
%Because the homology $diag(1,1,\zeta_n)\in Aut(C)$ for any $C\in\mathcal{M}_g^{pl}\left(diag(1,1,\zeta_{n})\right)$, we can consider the elements of this stratum to be of Type $n,\,(0,1)$. In particular, we conclude the result due to \cite[Theorem 7]{BaBa2}.
%\end{proof}

\begin{rem}\label{distinctfactors}
By non-singularity, the binary form $L_{d,Z}(X,Y)$ in Proposition \ref{geomcomp} can not have any repeated linear factors for any specializations of the parameters in $\mathbb{C}$. Otherwise, we may assume (up to $\mathbb{C}$-isomorphism via a change of the variable $X$) that the repeated factor is the line $X=0$. Therefore, $L_{d,Z}(X,Y)$ reduces to $X^2L_{d-2,Z}(X,Y)$. However, $d-1\notin S(u)_n$ for $u=1,2$, since $n>1$. Then, $F_{\mathcal{S}}(X,Y,Z)=0$ has the form $Z^2G(X,Y,Z)+X^2L_{d-2,Z}(X,Y)=0$, which is singular at the reference point $(0:1:0)$.
\end{rem}
%%%%%%%%%%%%%%%%%%%%%%%%%%%%%%%%%%%%%%%%%%%%%%%%
%%%%%%%%%%%%%%%%%%%%%%%%%%%%%%%%%%%%%%%%%%%%%%%%%%%%5

\subsection{Curves of $\widetilde{{(\mathcal{M}_{g}^{Pl})}}_{n,\operatorname{diag}}^{h}$ having odd signature}
\begin{defn}[signature]
Let $\phi: C\rightarrow C/G$ be a branched Galois covering
between smooth curves, and let $y_1,...,y_r$ be its branch points. The
signature of $\phi$ is defined as $(g_0; m_1,..., m_r)$, where $g_0$ is the genus of $C/G$ and
$m_i$ is the ramification index of any point in $\phi^{-1}(y_i)$.
\end{defn}
R. Hidalgo \cite{Hidalgo} considered complex curves $C$ such that the natural cover $\pi_C:\,C\rightarrow\,C/\operatorname{Aut}(C)$ has signature of the form $(0;m_1,m_2,m_3,m_4)$, proving that $C$ can be defied over its field of moduli if $m_4\notin\{m_1,m_2,m_3\}$. Artebani-Quispe in \cite{SAUL1} extended such a result to smooth curves of odd signature:

\begin{defn}[odd Signature]\label{defnoddsign}
A smooth curve $C$ of genus $g\geq2$ has odd signature if
the signature of the natural covering $\pi_C: C\rightarrow C/\operatorname{Aut}(C)$ is of the form $(0;m_1,...,m_r)$
where some $m_i$ appears exactly an odd number of times.
\end{defn}
\begin{thm}[Artebani-Quispe, Theorem 2.5, \cite{SAUL1}]\label{oddsign1}
Let $C$ be a smooth curve of genus $g\geq2$ defined over an algebraically closed field $F$. Let $L$ be a subfield of $F$, such that $F/L$ is Galois. If $C$ is an odd signature curve, then the field of moduli $M_{F/L}(C)$ is a field of definition for $C$.
\end{thm}

\begin{cor}\label{oddsign2}
Let $\mathcal{S}$ be a plane Riemann surface in $\widetilde{{(\mathcal{M}_{g}^{Pl})}}_{n,\operatorname{diag}}^{h}$ with $\operatorname{Aut}_+(\mathcal{S})$ cyclic of order $n$. Then, $\mathcal{S}$ has odd signature if and only if $n$ is odd and equals to either $d$ or $d-1$. Moreover, in this case that $\mathcal{S}$ is of odd signature, $F_{\mathcal{S}}(X,Y,Z)=0$ is defined by an equation of the form $Z^d+L_{d,Z}=0$ if $n=d$, and $Z^{d-1}Y+L_{d,Z}=0$ if $n=d-1$.
\end{cor}

\begin{proof}
By Remark \ref{distinctfactors}, we know that the binary form $L_{d,Z}$ factors into $d$ distinct factors associated to $d$ distinct roots, say $(a_i:b_i)\in\mathbb{P}^1_{\mathbb{C}}$, for $i=1,2,...,d$. Since $\operatorname{Aut}_+(F_{\mathcal{S}})=\langle\operatorname{diag}(1,1,\zeta_n)\rangle$, the covering $\pi_{\mathcal{S}}:\mathcal{S}\mapsto \mathcal{S}/\operatorname{Aut}_+(\mathcal{S})$ is ramified exactly at the $d$ points $\{(a_i:b_i:0)\}$ when $n|d$, plus the extra point $(0:0:1)$ when $n|d-1$. This gives $d$ branch points (resp. $d+1$) with ramification index $n$ when $n|d$ (resp. $n|d-1$). Consequently, $F_{\mathcal{S}}(X,Y,Z)=0$ has odd signature only if $n\,|\,d$ and $d$ odd or $n\,|\,d-1$ and $d$ even. The Riemann-Hurwitz formula reads as
$$(d-1)(d-2)-2=n(2g_0-2+d(1-1/n)),$$
when $n|d$ and $d$ odd, and as
$$(d-1)(d-2)-2=n(2g_0-2+(d+1)(1-1/n)),$$
when $n|d-1$ and $d$ even.  Setting $g_0=0$ and solving for $n$, we obtain $n=d$ (resp. $n=d-1$). This proves the "if and only if" statement.

Lastly, for $n=d$ (resp. $n=d-1$), the index set $S(1)_d$ (resp. $S(2)_{d-1}$) is obviously empty, and we deduce the defining equation for $\mathcal{S}$.
  \end{proof}
%%%%%%%%%%%%%%%%%%%%%%%%%%%%%%%%%%%%%%%%%%%%%%%%%%%%%%%%%%

\subsection{The stratum $\widetilde{{(\mathcal{M}_{g}^{Pl})}}_{n,\operatorname{diag}}^{h}$ with $n|\,d-1$}
Suppose that $\widetilde{{(\mathcal{M}_{g}^{Pl})}}_{n,\operatorname{diag}}^{h}$ is non-empty for some fixed integer $n>1$ with $n|d-1$. Let $\mathcal{S}\in{(\mathcal{M}_{g}^{Pl})}_{n,\operatorname{diag}}^{h}$ with cyclic automorphism group. Fix a non-singular plane model $F_{\mathcal{S}}(X,Y,Z)=0$ of degree $d$ over $\mathbb{C}$ in the family $\mathcal{C}_2$ of Proposition \ref{geomcomp}, such that $\operatorname{Aut}_+(F_{\mathcal{S}})=\langle\operatorname{diag}(1,1,\zeta_n)\rangle.$

Let us consider the subfamilies
\begin{eqnarray*}
\mathcal{C}^{(1)}_{2}&:&\,Z^{d-1}Y+\sum_{j\in S(2)_n}Z^{d-j}L_{j,Z}+X^d+X^{d-2}Y^2+\sum_{j=3}^{d}\,a_jX^{d-j}Y^{j}=0,\\
\mathcal{C}^{(3)}_{2}&:&\,Z^{d-1}Y+\sum_{j\in S(2)_n}Z^{d-j}L_{j,Z}+X^d+Y^d+\sum_{j=3}^{d-1}\,a_jX^{d-j}Y^{j}=0,\\
\mathcal{C}^{(4)}_{2}&:&\,Z^{d-1}Y+\sum_{j\in S(2)_n}Z^{d-j}L_{j,Z}+X^d+XY^{d-1}+\sum_{j=3}^{d-2}\,a_jX^{d-j}Y^{j}=0,\\
\mathcal{C}^{(5)}_{2}&:&\,Z^{d-1}Y+\sum_{j\in S(2)_n}Z^{d-j}L_{j,Z}+X^{d-1}Y+\sum_{j=3}^{d}\,a_jX^{d-j}Y^{j}=0.
\end{eqnarray*}
\begin{lem}\label{geometcomp2}
We have $$\mathcal{C}_2=\bigcup_{s=1}^{4}\mathcal{C}^{(s)}_{2}.$$
\end{lem}
\begin{proof}
By non-singularity, $X^d$ or $X^{d-1}Y$ should appear in $L_{d,Z}$. Hence, up to rescaling the variable $X$ and then renaming the parameters, we can split up $\mathcal{C}_2$, in Proposition \ref{geomcomp}, into two substrata defined by the equations
$$
Z^{d-1}Y+\sum_{j\in S(2)_n}Z^{d-j}L_{j,Z}+X^{d-1}Y+a_2X^{d-2}Y^2+...+a_{d-1}XY^{d-1}+a_dY^d=0
$$
along with,
$$
Z^{d-1}Y+\sum_{j\in S(2)_n}Z^{d-j}L_{j,Z}+X^d+a_1X^{d-1}Y+...+a_{d-1}XY^{d-1}+a_dY^d=0.
$$
% due to rescaling x,y,z we can assume the coefficient to be 1
We always can assume $a_2=0$ in the first component, by a change of variables of the shape $X\mapsto X-\frac{a_2}{d-1}Y$ and then renaming the parameters in order to obtain the fourth component $\mathcal{C}^{(4)}_2$. In the same way, we may set $a_1=0$ in the second component through $X\mapsto X-\frac{a_1}{d}Y$ and renaming the parameters next. Moreover, if $a_2=0$, then we split it up according to whether the monomial $Y^d$ appears or not (if it does not appear, then $XY^{d-1}$ does, by non-singularity). Therefore, we get $\mathcal{C}^{(2)}_{2}$ and $\mathcal{C}^{(3)}_{2}$ respectively, up to rescaling $Y$ and $Z$. Finally, when $a_2\neq0$, we rescale $Y$ and $Z$ to have $\mathcal{C}^{(1)}_{2}$.
\end{proof}
\begin{rem}\label{importantrema}
A priori, the index set $S(2)_n$ is empty if and only if $\operatorname{diag}(1,1,\zeta_{d-1})\in\operatorname{Aut}_+(\mathcal{C}_2)$. In this case, $n=d-1$ and  the subfamily $\mathcal{C}^{(4)}_2$ is not irreducible anymore, since it factors as $Y\cdot\,G(X,Y,Z)=0$. Hence, $\mathcal{C}_2$ reduces to $\bigcup_{s=1}^{3}\mathcal{C}^{(s)}_{2}$ if $n=d-1$.
\end{rem}

\begin{thm}\label{6.5}
Let $\mathcal{S}\in\widetilde{{(\mathcal{M}_{g}^{Pl})}}_{n,\operatorname{diag}}^{h}$ with $n\,|\,d-1$ and $\operatorname{Aut}_+(\mathcal{S})\simeq\Z/n\Z$. If $\mathbb{R}$ is the field of moduli for $C$,
relative to the Galois extension $\mathbb{C}/\mathbb{R}$, then it is also a field of definition.
\end{thm}

\begin{proof}

By Proposition \ref{geomcomp}, we may take a non-singular plane model $F_{\mathcal{S}}(X,Y,Z)=0$ for $\mathcal{S}$ in $\mathbb{P}^2_{\mathbb{C}}$ of degree $d$ in the family $\mathcal{C}_2$, such that $\operatorname{Aut}_+(F_{\mathcal{S}})=\langle\operatorname{diag}(1,1,\zeta_n)\rangle$. Hence $N_{\operatorname{Aut}_+(F_{\mathcal{S}})}(\mathbb{C})$ equals to $\operatorname{GL}_{2,Z}(\mathbb{C})$, using Proposition \ref{normalizerdetyermination}. Because of the monomial term $Z^{d-1}Y$ in the defining equation for $F_{\mathcal{S}}(X,Y,Z)=0$, the action of the normalizer $\operatorname{GL}_{2,Z}(\mathbb{C})$ is trivial except possibly an isomorphism of the shape $$[\alpha X+\beta Y:\gamma Y:Z]\in\operatorname{PGL}_3(\mathbb{C}).$$

We will show that $\beta=0$ and so $\mathcal{S}$ is isomorphic to its complex conjugate $^{\sigma}\mathcal{S}$ via an isomorphism $\phi_{\sigma}=\operatorname{diag}(1,\lambda,\mu)$ where $\lambda$ and $\mu$ are roots of unity. In particular, $\phi_{\sigma}$ satisfies the Weil's condition of descent ( $\phi_{\sigma}\circ\,^{\sigma}\phi_{\sigma}=1$), and $\mathbb{R}$ is a field of definition, see Theorem \ref{weilcocylcle}. First, by construction, the union decomposition $\mathcal{C}_2=\bigcup_{s=1}^{4}\mathcal{C}^{(s)}_{2}$ in Lemma \ref{geometcomp2} is well-defined up to $\mathbb{C}$-isomorphisms, that is even $[\alpha X+\beta Y:\gamma Y:Z]$ does not define an isomorphism between two curves in two distinct components. Second,  if $C$ belongs to any of the subfamilies $\mathcal{C}^{(s)}_{2}$ for $s=1,2,3$, then $\alpha^{d-1}\beta=0$, since $X^{d-1}Y$ does not appear in the defining equation $F_{\mathcal{S}}(X,Y,Z)=0$. Thus $\beta=0$, because  $[\alpha X+\beta Y:\gamma Y:Z]$ must be invertible. We mimic the argument by switching the monomial term $X^{d-1}Y$ with $X^{d-2}Y^2$ for $\mathcal{C}^{(4)}_{2}$.
\end{proof}
%%%%%%%%%%%%%%%%%%%%%%%%%%%%%%%%%%%%%%%%%%%%%%%%%%%%%%%%%%%%%%%%%%%%%%%%%%%%%%%%%%%%%%%%%%%%

\subsection{The stratum $\widetilde{{(\mathcal{M}_{g}^{Pl})}}_{n,\operatorname{diag}}^{h}$ with $n|\,d$ and $d$ odd}
\begin{thm}\label{coroddsign}
Let $\mathcal{S}\in\widetilde{{(\mathcal{M}_{g}^{Pl})}}_{n,\operatorname{diag}}^{h}$ with $n\,|\,d$ and $\operatorname{Aut}_+(\mathcal{S})\simeq\Z/n\Z$, where $d\geq5$ is odd.
If $\mathbb{R}$ is the field of moduli for $\mathcal{S}$, relative to the Galois extension $\mathbb{C}/\mathbb{R}$, then it is also a field of definition.
\end{thm}
\begin{proof}
Theorem \ref{oddsign1} and Corollary \ref{oddsign2} gives the result when $\mathcal{S}\in\widetilde{{(\mathcal{M}_{g}^{Pl})}}_{d,\operatorname{diag}}^{h}$, i.e. when $n=d$. Therefore, we can suppose that $1<n<d$. By Proposition \ref{geomcomp}, we may take a non-singular plane model $F_{\mathcal{S}}(X,Y,Z)=0$ in $\mathbb{P}^2_{\mathbb{C}}$ in the family $\mathcal{C}_1$, whose automorphism group $\operatorname{Aut}_+(F_{\mathcal{S}})=\langle\operatorname{diag}(1,1,\zeta_n)\rangle$. More precisely, we have an equation of the form
$$\mathcal{C}'_1:\,Z^d+\sum_{1\leq t<d/n}Z^{d-tn}L_{tn,Z}+L_{d,Z}=0,$$
such that $L_{tn,Z}\neq0$ for some $1\leq t<d/n$.

We first show the next observation:

\noindent\textit{Observation.} Let $\mathcal{C}'_{1,0}:Z^d+L_{d,Z}=0$. Then, for any isomorphism $$\phi_{\sigma}:\,^{\sigma}\mathcal{C}'_{1,0}\rightarrow\mathcal{C}'_{1,0},$$
we can find $\eta_{\sigma}\in\langle\operatorname{diag}(\zeta_d^{-1},\zeta_d^{-1},1)\rangle$ and an isomorphism $\widetilde{\phi}_{\sigma}:\,^{\sigma}\mathcal{C}'_{1}\rightarrow\mathcal{C}'_{1}$, such that $\eta_{\sigma}\circ\phi_{\sigma}$ and $\widetilde{\phi}_{\sigma}$ have the same action on $\mathcal{C}'_{1}$.
\begin{proof}
Because $N_{\operatorname{Aut}_+(F_{\mathcal{S}})}(\mathbb{C})=\operatorname{GL}_{2,Z}(\mathbb{C})$ (Proposition \ref{normalizerdetyermination}), then the geometric fibers which are isomorphic are obtained through the action of $\operatorname{GL}_{2,Z}(\mathbb{C})$. It is also clear that an element $\phi\in\operatorname{GL}_{2,Z}(\mathbb{C})$, which acts non-trivially on the family $\mathcal{C}'_{1,0}$, also has non-trivial action on $\mathcal{C}'_1$. The converse is true, unless $\phi\in\langle\operatorname{diag}(\zeta_d^{-1},\zeta_d^{-1},1)\rangle$. Therefore, the number of isomorphic geometric fibers in $\mathcal{C}'_{1,0}$ is exactly the same number of those in the family $\mathcal{C}'_1$, coming by the action of $\operatorname{GL}_{2,Z}(\mathbb{C})\setminus\langle\operatorname{diag}(\zeta_d^{-1},\zeta_d^{-1},1)\rangle$. Consequently, the action of any isomorphism $\phi_{\sigma}:\,^{\sigma}\mathcal{C}'_{1,0}\rightarrow\mathcal{C}'_{1,0}$ can always be extended to an action $\widetilde{\phi}_{\sigma}:\,^{\sigma}\mathcal{C}'_{1}\rightarrow\mathcal{C}'_{1}$. In particular, the composition $\widetilde{\phi}_{\sigma}\circ\phi_{\sigma}^{-1}$ acts trivially on $Z^d+L_{d,Z}=0$, that is $\widetilde{\phi}_{\sigma}\circ\phi_{\sigma}^{-1}=\eta_{\sigma}\in\operatorname{Aut}_+(\mathcal{C}'_{1,0})=\langle\operatorname{diag}(\zeta_d^{-1},\zeta_d^{-1},1)\rangle$.
\end{proof}
Next, by the aid of Theorem \ref{oddsign1} and Corollary \ref{oddsign2}, we have an isomorphism $\phi_{\sigma}:\,^{\sigma}\mathcal{C}'_{1,0}\rightarrow\mathcal{C}'_{1,0}$ in $\operatorname{GL}_{2,Z}(\mathbb{C})$, satisfying the Weil's cocycle condition of descent ($\phi_{\sigma}\circ\,^{\sigma}\phi_{\sigma}=1$), see Theorem \ref{weilcocylcle}. Using the previous observation, we also have an isomorphism $\widetilde{\phi}_{\sigma}:=\eta_{\sigma}\circ\phi_{\sigma}:\,^{\sigma}\mathcal{C}'_{1}\rightarrow\mathcal{C}'_{1}$ in $\operatorname{GL}_{2,Z}(\mathbb{C})$, where $\eta_{\sigma}:=\operatorname{diag}(\epsilon_{\sigma}^{-1},\epsilon_{\sigma}^{-1},1)$ for some $d$-th root of unity $\epsilon_{\sigma}$. Hence, it satisfies  $$(^{\sigma}L_{tn,Z})(X,Y)=L_{tn,Z}(\widetilde{\phi}_{\sigma}(X,Y))=\epsilon^{-tn}_{\sigma}L_{tn,Z}(\phi_{\sigma}(X,Y)),$$
for all $1\leq t<d/n$. That is,
\begin{eqnarray*}
L_{tn,Z}(X,Y)&=&(^{\sigma^2}L_{tn,Z})(X,Y)=\,^{\sigma}(\epsilon^{-tn}_{\sigma}L_{tn,Z}(\phi_{\sigma}(X,Y)))\\
&=&\sigma(\epsilon^{-tn}_{\sigma})\,^{\sigma}(L_{tn,Z})(^{\sigma}\phi_{\sigma}(X,Y)):=\sigma(\epsilon^{-tn}_{\sigma})\,^{\sigma}(L_{tn,Z})(X',Y')\\
&=&\sigma(\epsilon^{-tn}_{\sigma})\epsilon^{-tn}_{\sigma}L_{tn,Z}(\phi_{\sigma}(X',Y'))=(\epsilon_{\sigma}\,\sigma(\epsilon_{\sigma}))^{-tn}L_{tn,Z}((\phi_{\sigma}\circ\,^{\sigma}\phi_{\sigma})(X,Y))\\
&=&L_{tn,Z}(((\eta_{\sigma}\circ\,^{\sigma}\eta_{\sigma})\circ(\phi_{\sigma}\circ\,^{\sigma}\phi_{\sigma}))(X,Y))\\
&=&L_{tn,Z}(((\eta_{\sigma}\circ\phi_{\sigma})\circ(^{\sigma}\eta_{\sigma}\circ\,^{\sigma}\phi_{\sigma}))(X,Y))\\
&=&L_{tn,Z}((\widetilde{\phi}_{\sigma}\circ\,^{\sigma}\widetilde{\phi}_{\sigma})(X,Y)).
\end{eqnarray*}
So $\widetilde{\phi}_{\sigma}$ satisfies the Weil's condition of descent as well, and
$\mathbb{R}$ is a field of definition.
\end{proof}

\subsection{The stratum $\widetilde{{(\mathcal{M}_{g}^{Pl})}}_{n,\operatorname{diag}}^{h}$ with $n\,|\,d$ and $d$ even}\label{devensection}
There is no plane Riemann surfaces $\mathcal{S}$ of genus $3$ with automorphism group $\operatorname{PGL}_3(\mathbb{C})$-conjugate to $\langle \operatorname{diag}(1,1,\zeta_4)\rangle$, see \cite{He} or \cite{Bars} for more details.. Hence, the stratum $\widetilde{{(\mathcal{M}_{3}^{Pl})}}_{4,\operatorname{diag}}^{h}$ is empty, and we have nothing to say in this case.

Take $m,r\in\mathbb{N}$ such that $2mr>5$ and $r$ is odd when $m$ does. Consider a binary form $G(X,Y)\in\mathbb{C}[X,Y]\setminus\mathbb{R}[X,Y]$ given by
$$G(X,Y):=\prod_{i=1}^{r} (X^m-a_iY^m)(X^m+\,^{\sigma}a_iY^m),$$
for some $a_1,...,a_r\in\mathbb{C}$ such that the next conditions hold: $G(X,1)$ has no repeated zeros, the map $[\alpha:\beta]\mapsto[\beta:\alpha]$ does not map the zero set of $G(X,Y)$ into itself, for any root of unity $\zeta$ we should have $\{a_i,-1/a_i^c\}\neq\{\zeta a_i,-\zeta/a_i^c\}$, and when $m=3$, the map $[\alpha:\beta]\mapsto[-\alpha+(1+\sqrt{3})\beta:(1+\sqrt{3})\alpha+\beta]$ does not map the zero set of $G(X,Y)$ into itself.
\begin{prop}(B. Huggins, \cite[Chapter 7, \S1]{Hug})\label{hugginsexample}
Following the above notations, let $\mathcal{S}$ be a plane Riemann surface of degree $>5$ given in $\mathbb{P}^2_{\mathbb{C}}$ by an equation of the form
$$F_{\mathcal{S}}(X,Y,Z):=Z^{2mr}-G(X,Y)=0$$
Then, the automorphism group $\operatorname{Aut}_+(F_{\mathcal{S}})$ is diagonal and equals $$\langle\operatorname{diag}(\zeta_m,1,1),\,\operatorname{diag}(1,\zeta_m,1),\,\operatorname{diag}(1,1,\zeta_{2mr})\rangle.$$
Moreover, the field of moduli $M_{\mathbb{C}/\mathbb{R}}(\mathcal{S})=\mathbb{R}$, but it is not a field of definition.
\end{prop}

\begin{cor}\label{evendegreeexample}
For any integer $d=2(2k+1)\geq6$, pseudoreal-plane Riemann surfaces in $\widetilde{{(\mathcal{M}_{g}^{Pl})}}_{d,\operatorname{diag}}^{h}$ exist. In other words,
there exist pseudoreal-plane Riemann surfaces $\mathcal{S}$ of genus $g=(d-1)(d-2)/2$ with $\operatorname{Aut}_+(\mathcal{S})$ conjugated to $\langle\operatorname{diag}(1,1,\zeta_{d})\rangle$.
\end{cor}

\begin{proof}
Let $m=1$ and $r=2k+1$ in Proposition \ref{hugginsexample}.
\end{proof}
\subsubsection{The stratum $\widetilde{{(\mathcal{M}_{g}^{Pl})}}_{\frac{d}{p},\operatorname{diag}}^{h}$ with $d$ even and $p\,|\,d$ is prime}
It remains yet the study of $\widetilde{{(\mathcal{M}_{g}^{Pl})}}_{n,\operatorname{diag}}^{h}$ when the degree $d\geq4$ is even and $n$ divides $d$ properly. In this case, the field of moduli does not need to be a field of definition. The first (explicit) example appears for genus $g=3$ curves, and we address the reader to \cite[\S4, Lemma 4.2, Proposition 4.3]{SAUL1}, for a smooth plane quartic curve over $\mathbb{C}$, not definable over its field of moduli $\mathbb{R}$, and
whose conformal automorphism group is $\Z/2\Z$. In what follows, we generalize this construction for higher degrees:
\begin{example}\label{examplesndividesproperlydeven}
Take $d=2pm$ an integer with $m\geq3$ odd and $p$ is a prime number. Define a Riemann surface $\mathcal{S}$ in $\mathbb{P}^2_{\mathbb{C}}$ by an equation of the form
\begin{eqnarray}\label{example101}
F_{\mathcal{S}}(X,Y,Z):=Z^{d}+Z^{\frac{d}{p}}g(X,Y)-f(X,Y)=0,
\end{eqnarray}
where
\begin{eqnarray*}
g(X,Y)&:=&\prod_{i=1}^{\frac{(p-1)d}{2p}}(X-a_iY)(X+\frac{1}{a_i}Y)\\
f(X,Y)&:=&\prod_{i=1}^{\frac{d}{2}}(X-b_iY)(X+1/^{\sigma}b_iY),
\end{eqnarray*}
with $a_i\in\mathbb{R}\setminus\{0\}$, for $1\leq i\leq (p-1)d/2p$, and $b_i\in\mathbb{C}\setminus\{0\}$, for $1\leq i\leq d/2$.
Suppose also that $f(X,Y)$ and $g(X,Y)$ have no repeated zeros. We also choose the $a_i's$ in the way that $g(X,Y)$ is not $\psi_c$-invariant or $\psi_{a,b}$-invariant for any $\psi_c:(X:Y)\mapsto(Y:cX)$ and $\psi_{a,b}:(X:Y)\mapsto(X+aY:bX-Y)$ in $\operatorname{PGL}_{2}(\mathbb{C})$.
\begin{thm}
For any integer $d=2pm$ with $m\geq3$ odd and $p$ is prime, pseudoreal-plane Riemann surfaces $\mathcal{S}$ in $\widetilde{{(\mathcal{M}_{g}^{Pl})}}_{\frac{d}{p},\operatorname{diag}}^{h}$ exist. That is,
there exist pseudoreal-plane Riemann surfaces of genus $g=(d-1)(d-2)/2$ with conformal automorphism group conjugated to $\langle\operatorname{diag}(1,1,\zeta_{d/p})\rangle$.
\end{thm}

\begin{proof}
 Let $\mathcal{S}$ be a complex curve defined by an equation of the form (\ref{example101}) such that $\prod_{i=1}^{\frac{d}{2}}b_i\in\mathbb{R}$.
We first show that $\mathcal{S}$ has no singular points in $\mathbb{P}^2_{\mathbb{C}}$.
Since $F_{\mathcal{S}}(X,0,Z)=Z^d+(X^{p-1}Z)^{\frac{d}{2}}-X^d=0$ has no repeated zeros, the common zeros of $F_{\mathcal{S}}(X,0,Z)$ and $\frac{\partial F_{\mathcal{S}}}{\partial X}(X,0,Z)$ do not exist.
Moreover, $\frac{\partial F_{\mathcal{S}}}{\partial X}(X,1,Z)=Z^{\frac{d}{p}}g'(X,1)-f'(X,1)$ and
$\frac{\partial F_{\mathcal{S}}}{\partial Z}(X,1,Z)=\frac{d}{p}Z^{\frac{d}{p}-1}(pZ^{\frac{(p-1)d}{p}}+g(X,1))$. But $f(X,Y)$ is square free, then $(X:1:0)$ gives no singularities on $F_{\mathcal{S}}(X,Y,Z)=0$.
Furthermore, if we substitute $g(X,1)=-pZ^{\frac{(p-1)d}{p}}$ into $F_{\mathcal{S}}(X,1,Z)=\frac{\partial F_{\mathcal{S}}}{\partial X}(X,1,Z)=0$, we get
% substitute into F_{\mathcal{S}}(X,1,Z)=0 to get f(X,1)=(1-p)Z^d. From \frac{\partial F_{\mathcal{S}}}{\partial X}(X,1,Z)=0 after raising to power p, we also get Z^dg'(X,1)^p=f'(X,1)^p
% Hence f(X,1)g'(X,1)^p=(1-p)Z^dg'(X,1)^p=(1-p)f'(X,1)^p
that $\mathcal{S}$ is singular only if $f(X,1)g'(X,1)^p=(1-p)f'(X,1)^p$, that is when $f(X,1)$ has repeated zeros, a contradiction.
Hence $\mathcal{S}$ is plane.

Second, one easily checks that $F_{\mathcal{S}}(X,Y,Z)=0$ is isomorphic to its complex conjugate $(^{\sigma}F_{\mathcal{S}})(X,Y,Z)=0$ via the isomorphism $\phi_{\sigma}:=[-Y:X:\zeta_{2d}^pZ]$.
Consequently, $\mathbb{R}$ is the field of moduli for $\mathcal{S}$, relative to $\mathbb{C}/\mathbb{R}$.
If we assume that our claim on $\operatorname{Aut}_+(\mathcal{S})$ is true, then $\mathbb{R}$ is not a field of definition for $\mathcal{S}$.
To see this, let $\phi_{\sigma}'$ be any isomorphism. Then, $\phi_{\sigma}\circ\phi_{\sigma}'^{-1}\in\operatorname{Aut}_+(\mathcal{S})$, and so
$\phi_{\sigma}'=\phi_{\sigma}\circ\operatorname{diag}(1,1,\zeta_{d/p})^{r}$ for some integer $0\leq r<d/p$. However, any such $\phi_{\sigma}'$ does not satisfy Weil's condition of descent ($\phi_{\sigma}'\circ\,^{\sigma}\phi_{\sigma}'=1$) in Theorem \ref{weilcocylcle}, since $\phi_{\sigma}'\circ\,^{\sigma}\phi_{\sigma}'=\operatorname{diag}(1,1,-1)\neq1$. Thus $\mathcal{S}$ is pseudoreal.

Now, we prove the claim on $\operatorname{Aut}_+(\mathcal{S})$ by using quite similar techniques as in \cite{BaBa1, BaBa3, BaBa2}.
Obviously, $\psi:=\operatorname{diag}(1,1,\zeta_{d/p})\in\operatorname{Aut}_+(F_{\mathcal{S}})$ is an homology of order $d/p\geq4$.
Therefore, $\operatorname{Aut}_+(F_{\mathcal{S}})$ fixes a point, a line or a triangle, by Theorem \ref{homologies}.
In particular, it is not conjugate to any of the finite primitive group mentioned in Theorem \ref{Harui,Mitchell}-(iii).

Now, we treat each of the following subcases:
\begin{enumerate}[(i)]
\item A line $L\subset\mathbb{P}^{2}_{\mathbb{C}}$ and a point $P\notin L$ are leaved invariant: By \cite[Theorem 2.1]{Ha}, we can think about $\operatorname{Aut}_+(F_{\mathcal{S}})$ in a short exact sequence
\scriptsize
$$
\xymatrix
{
1\ar[r]  & \mathbb{C}^*\ar[r]                    & \operatorname{GL}_{2,Y}(\mathbb{C})\ar[r]^{\varrho}& \operatorname{PGL}_2(\mathbb{C})\ar[r]& 1         \\
         &                              &                                        &                              &\\
1\ar[r]  & \langle\psi\rangle\ar[r]\ar@{^{(}->}[uu] & \operatorname{Aut}_+(F_{\mathcal{S}})\ar[r]\ar@{^{(}->}[uu] & G\ar[r]\ar@{^{(}->}[uu]& 1
}
$$
\normalsize
where $G$ is conjugate to a cyclic group $\mathbb{Z}/m\mathbb{Z}$ of
order $m\leq d-1$, a Dihedral group $\operatorname{D}_{2m}$ of order $2m$
with $m|(d-2)$, one of the alternating groups $\operatorname{A}_4$, $\operatorname{A}_5$, or to
the symmetry group $\operatorname{S}_4$. Any such $G$, which is not cyclic, contains an element of order $2$.
Let $\psi'\in\operatorname{Aut}_+(F_{\mathcal{S}})$ such that $\varrho(\psi')$ has order $2$. Then, $\varrho(\psi')$ has the shape $\psi_c$ or $\psi_{a,b}$
for some $a,b,c\in\mathbb{C}\setminus\{0\}$, which is absurd by our assumptions on $g(X,Y)$.
Consequently, $G=\varrho(\operatorname{Aut}_+(F_{\mathcal{S}}))$ is cyclic, generated by the image of a specific $\psi_G\in \operatorname{GL}_{2,Y}(\mathbb{C})$.
This would lead to a polynomial expression of $b_i's$ in terms of the $a_i's$, hence we still have infinitely many possibilities to choose the $b_i's$ such that
$f(X,Y)$ not $\langle\varrho(\psi_G)\rangle$-invariant. In particular, $|G|=1$ and $\operatorname{Aut}_+(F_{\mathcal{S}})$ is $\operatorname{PGL}_3(\mathbb{C})$-conjugate to $\langle\operatorname{diag}(1,1,\zeta_{d/p})\rangle$.
\item A triangle $\Delta$ is fixed by $\operatorname{Aut}_+(F_{\mathcal{S}})$ and neither a line nor a point is leaved invariant:
It follows by the proof of Theorem 2.1 in \cite{Ha} that $(\mathcal{S},\operatorname{Aut}_+(\mathcal{S}))$ should be a descendant of the Fermat curve $F_d$
or the Klein curve $K_d$ as in Theorem \ref{Harui,Mitchell}. Note that $d/p$ does not divide $|\operatorname{Aut}_+(K_d)|=3(d^2-3d+3)$,
    e.g. \cite[Propositions 3.5]{Ha}. Therefore, $(\mathcal{S},\operatorname{Aut}_+(\mathcal{S}))$ is not a descendant of $K_d$.
    Hence, $\exists \phi\in\operatorname{PGL}_3(\mathbb{C})$ where $H:=\phi^{-1}\operatorname{Aut}_+(F_{\mathcal{S}})\phi\leq\operatorname{Aut}_+(F_d)$. It is also well known (e.g. \cite[Proposition 3.3]{Ha}) that $\operatorname{Aut}_+(F_d)$ is a semidirect product of $\operatorname{S}_3=\langle T:=[Y:Z:X],R:=[X:Z:Y]\rangle$ acting on $(\mathbb{Z}/d\mathbb{Z})^2=\langle[\zeta_dX:Y:Z],[X:\zeta_dY:Z]\rangle.$ Thus any element of $\phi^{-1}\operatorname{Aut}_+(\mathcal{S})\phi$ has the shape $DR^iT^j$, for some $0\leq i\leq1$ and $0\leq j\leq2$ and $D$ is of diagonal shape in $\operatorname{PGL}_3(\mathbb{C})$. It is straightforward to check that any $DT^j$ and $DRT^j$ with $j\neq0$ has order
    $3<d/p$. Thus $\phi^{-1}\psi\phi$ has also a diagonal shape, and then we may take $\phi$ in the normalizer of $\langle\psi\rangle$, up to a change of variables in $\operatorname{Aut}_+(F_d)$.  In this case, we can think about $\operatorname{Aut}_+(F_{\mathcal{S}})$ in the commutative diagram
    \scriptsize
    $$
    \xymatrix
    {
     1\ar[r]  & (\mathbb{Z}/d\mathbb{Z})^2\ar[r]                    & \operatorname{Aut}_+(F_d)\ar[r]^{\varrho}& \operatorname{S}_3\ar[r]& 1         \\
         &                              &                                        &                              &\\
    1\ar[r]  & \operatorname{Ker}(\varrho|_{H})=\langle\psi\rangle\ar[r]\ar@{^{(}->}[uu] & H\ar[r]\ar@{^{(}->}[uu] & G:=\operatorname{Im}(\varrho|_{H})\ar[r]\ar@{^{(}->}[uu]& 1
    }
    $$
\normalsize
% remark because P\in GL_2,z, we only have the powers of Z^d Z^{d/2} even in the transofrmaed form. This implies \operatorname{Ker}(\varrho|_{\operatorname{Aut}}(C))
The variable $Z$ in the transformed defining equation via $\phi$ appears exactly as the original equation in the statement. Hence, $G$ is at most cyclic of order $2$, since otherwise $H$ must have an element of the shape $[\zeta_d^sY:\zeta_d^tZ:X]$ or $[\zeta_d^sZ:\zeta_d^tX:Y]$ for some integers $s,t$, which is not possible. For the same reason, $G$ is then generated by a certain $\varrho([\zeta_d^mY:\zeta_d^nX:Z])$, and as before, it only requires to exclude finitely many possibilities of $\{b_i\}\subset\mathbb{C}\setminus\{0\}$ such that $f(\phi(X,Y))$ is not $\langle\varrho([\zeta_d^mY:\zeta_d^nX:Z])\rangle$-invariant, where $\phi$ is the restriction of $\phi$ on $\mathbb{C}[X,Y]$.
\end{enumerate}
 \end{proof}
%
%\begin{rem}
%The restriction on the zero set of $g(X,Y)$, not to be invariant under any $\psi_c$ or $\psi_{a,b}$, is not restrictive as it looks. It only imposes finitely many algebraic conditions on the $a_i's$: For instance, an $\psi_{c}$ acts as a product of pairwise disjoint $2$-cycles on the set $\{(a_i:1)\}_i$ since it has order $2$ in $\operatorname{PGL}_{2}(\mathbb{C})$. So if $\psi_c:(a_s:1)\leftrightarrow (a_t:1)$ (resp. $(-1/a_t:1)$) for some $s,t$, then $c=a_s/a_t$ (resp. $-a_t/a_s$). Therefore, it suffices to choose the $a_i's$ such that $\{a_i,\frac{-1}{a_i}\}_{i}\neq\{\frac{\pm a_t}{a_sa_i},\frac{\mp a_ta_i}{a_s}\}_i$ for any $s,t$. In this case, $g(X,Y)$ is not $\psi_c$-invariant for any $\psi_c\in\operatorname{PGL}_2(\mathbb{C})$.
%
%The action of an $\psi_{a,b}$ is treated in the same way. However, it is a bit more tedious.
%%Roughly speaking, $\psi_{a,b}:(a_s:1)\leftrightarrow (a_t:1)$ (resp. $(-1/a_t:1)$) for some $s,t$ gives $a=(ba_s-1)a_t-a_s$ (resp. $\frac{1-a_s(b+a_t)}{a_t}$). Moreover, since $|\{a_i,\frac{-1}{a_i}_i|\geg4$, then we can also assume that $(a_{s'}:1)\leftrightarrow (a_{t'}:1)$ (resp. $(-1/a_{t'}:1)$ This in turns gives four possible values of $b$ to be excluded
%\end{rem}
\end{example}
%\begin{rem}
%We mention here that the above example works if we replace $d$, divisible by $4$ with $d$, divisible by $2p$ for any prime $p$, such that $\frac{d}{p}\geq6$. And, all the discussion is quite similar.
%\end{rem}

%%%%%%%%%%%%%%%%%%%%%%%%%%%%%%%%%%%%%%%%%%%%%%%%%%%%%%%%%%%%%%%%%%%%%%%%%%%%%%%%%%%%%%%

\subsection{On the stratum for $\Z/2\Z\times\Z/2\Z$}\label{sectionremainingsirutations}
It remains the study when $\operatorname{Aut}_+(\mathcal{S})$ is conjugate to the Klein four group $\langle\operatorname{diag}(1,-1,1),\operatorname{diag}(1,1,-1)\rangle$.

We first generalize \cite[Lemma 10]{BaBa3}:
\begin{prop}
Let $\mathcal{S}$ be a plane Riemann surface of odd degree $d\geq5$. Then, $\operatorname{Aut}_+(\mathcal{S})$ can not be $\operatorname{PGL}_3(\mathbb{C})$-conjugate to a $\Z/2\Z\times\Z/2\Z,\,\operatorname{A}_4,\,\operatorname{S}_4$ or $\operatorname{A}_5$.
\end{prop}

\begin{proof}
By \cite{Ha, Mit}, a $\Z/2\Z\times\Z/2\Z\subset\operatorname{PGL}_3(\mathbb{C})$, giving invariant a non-singular plane model $F_{\mathcal{S}}(X,Y,Z)=0$ of degree $d\geq4$ over $\mathbb{C}$, should fix a point not lying on $\mathcal{S}$ or $(\mathcal{S},\operatorname{Aut}_+(S))$ must be a descendant of the Fermat curve $F_d:\,X^d+Y^d+Z^d=0$ or the Klein curve $K_d:\,X^{d-1}Y+Y^{d-1}Z+Z^{d-1}X=0$.
But if $d$ is odd, then $4$ does not divide $|\operatorname{Aut}_+(F_d)|=6d^2$ and $|\operatorname{Aut}_+(K_d)|=3(d^2-3d+3)$, e.g. \cite[Propositions $3.3,\,3.5$]{Ha}. In particular, $(\mathcal{S},\operatorname{Aut}_+(\mathcal{S}))$ can not be a descendant of $F_d$ or $K_d$, and we can think about $\Z/2\Z\times\Z/2\Z$, in a short exact sequence of the form $1\rightarrow N=1\rightarrow H\rightarrow H\rightarrow1$, where $H$ is $\operatorname{PGL}_2(\mathbb{C})$-conjugate to $\Z/2\Z\times\Z/2\Z$, by Harui's main result \cite[Theorem 2.1]{Ha}. Let $H=\langle\eta_1,\eta_2\rangle\leq\operatorname{PGL}_2(\mathbb{C})$ acts on the variables $Y,Z$, then we can assume, up to conjugation of groups in $\operatorname{PGL}_2(\mathbb{C})$, that $\eta_{1}=\operatorname{diag}(1,-1)$ and $\eta_2=[aY+bZ:cY-aZ]$. % recall that \eta_2 has oreder $2$
Because $\eta_1\eta_2=\eta_2\eta_1$, we get $\eta_2=diag(-1,1)$ or $[bZ:cY]$ for some $bc\neq0$. Therefore, $\mathcal{S}$ should have a non-singular plane model of the form
$Z^{d-1}L_{1,Z}+Z^{d-3}L_{3,Z}+...+Z^2L_{d-2,Z}+L_{d,Z}=0$, and $Y^{d-1}L_{1,Y}+Y^{d-3}L_{3,Y}+...+Y^2L_{d-2,Y}+L_{d,Y}=0$ simultaneously. This reduces $F_{\mathcal{S}}(X,Y,Z)$ to $X\,\cdot\,G(X,Y,Z)$ for some homogenous polynomial of degree $d-1$, a contradiction to non-singularity. Hence, there is no plane Riemann surface $\mathcal{S}$ of odd degree $d$ with $\Z/2\Z\times\Z/2\Z\leq\operatorname{Aut}_+(\mathcal{S})$. The other part of the statement is immediate, since any of these group contains a subgroup isomorphic to $\Z/2\Z\times\Z/2\Z$.
\end{proof}
Now, let $\mathcal{S}$ be a plane Riemann surface of even degree $d\geq4$ with $$\langle\operatorname{diag}(1,-1,1),\,\operatorname{diag}(1,1,-1)$$
acting on a non-singular plane model $F_{\mathcal{S}}(X,Y,Z)=0$ over $\mathbb{C}$. It is obvious that $F_{\mathcal{S}}(X,Y,Z)=0$ must be of the form 
\begin{equation}\label{geomcompletefamilyremaining}
X^d+Y^d+Z^d+\sum_{\tiny{\begin{array}{c}
                           0\leq s,t,u\leq (d/2)-1\\
                             s+t+u=d/2
                          \end{array}}}\,\alpha_{s,t,u}(X^sY^tZ^u)^2=0,
                          \end{equation}
                          for $\alpha_{s,t,u}\in\mathbb{C}$.
\begin{example}
The family $\mathcal{C}_{a,b,c}$ defined by
$$X^4+Y^4+Z^4+aX^2Y^2+bX^2Z^2+cY^2Z^2=0,$$
with $a^2,b^2,c^2$ are pairwise distinct, and $a^2+b^2+c^2-abc,a^2,b^2,c^2\neq4$, has automorphism group $\Z/2\Z\times\Z/2\Z$.
\begin{thm}(Artebani-Quispe, \cite[\S 4]{SAUL1})\label{thmlastcase}
Let $\mathcal{S}$ be plane Riemann surface in the family $\mathcal{C}_{a,b,c}$ above. Then, $\mathcal{S}$ is not pseudoreal.
   \end{thm}
\begin{rem}
Let $G$ be the group acting on the triples $(a,b,c)\in\mathbb{C}^3$, generated by
\begin{multicols}{2}
$$
g_1(a,b,c):=(b,a,c),
$$
$$
g_2(a,b,c):=(b,c,a),
$$

$$
g_3(a,b,c):=(-a,-b,c),
$$

$$
 g_4(a,b,c):=(a,-b,-c).
$$
\end{multicols}
E. W. Howe \cite[Proposition 2]{howe} observed that any isomorphism between $\mathcal{C}_{a,b,c}$ and $\mathcal{C}_{g(a,b,c)}$, for $g\in G$, is defined over $\mathbb{Q}(i)$.
\end{rem}
\end{example}
We will show that Theorem \ref{thmlastcase} is true in general:
\begin{thm}
Let $\mathcal{S}$ be a plane Riemann surface of even degree $d\geq4$ whose conformal automorphism group is $\operatorname{PGL}_{3}(\mathbb{C})$-conjugate to $\langle\operatorname{diag}(1,-1,1),\,\operatorname{diag}(1,1,-1)\rangle.$
Then, $\mathcal{S}$ is not pseudoreal.                        \end{thm}

\begin{proof}
One easily checks that the normalizer of $\langle\operatorname{diag}(1,-1,1),\,\operatorname{diag}(1,1,-1)\rangle$ in $\operatorname{PGL}_3(\mathbb{C})$ is $N=\langle[X:Z:Y],\,[Y:Z:X],\,\operatorname{D}(\mathbb{C})\rangle$. If $M(\mathbb{C}/\mathbb{R})=\mathbb{R}$, then we must have an isomorphism $\phi_{\sigma}:\,^{\sigma}\mathcal{S}\rightarrow\mathcal{S}$ in $N$. Consequently, $\phi_{\sigma}\circ\,^{\sigma}\phi_{\sigma}\in\operatorname{Aut}_+(\mathcal{S})$, and hence $\phi_{\sigma}$ reduces to one of the shapes $$\operatorname{diag}(1,\lambda,\mu),\,[X:\lambda Z:\mu Y],\,[\lambda Y:\mu X:Z],\,\, \text{or}\,\,[\lambda Z:Y:\mu X],$$
where $\lambda,\mu\in\mathbb{C}\setminus\{0\}$. Moreover, $F_{\mathcal{S}}(X,Y,Z)=0$ has core $X^d+Y^d+Z^d$ as in (\ref{geomcompletefamilyremaining}), thus $\lambda,\mu$ are $d$-th root of unity in $\mathbb{C}$. Now, we treat each of the above cases:

If $\phi_{\sigma}=\operatorname{diag}(1,\lambda,\mu)$ with $\lambda$ and $\mu$ are roots of unity in $\mathbb{C}$, then $\phi_{\sigma}\circ\,^{\sigma}\phi_{\sigma}=1$. That is, $\phi_{\sigma}$ satisfies Weil's condition of decent in Theorem \ref{weilcocylcle}, and $\mathbb{R}$ is a field of definition for $\mathcal{S}$.

If $\phi_{\sigma}=[X:\lambda Z:\mu Y]$, then $F_{\mathcal{S}}(X,Y,Z)=0$ has the form
\begin{equation}\label{geomcompletefamilyremaining2}
X^d+Y^d+Z^d+\sum_{\tiny{\begin{array}{c}
                           0\leq s,t,u\leq (d/2)-1\\
                             s+t+u=d/2
                          \end{array}}}\,X^{2s}(YZ)^{2t}(\alpha_{s,t,u}Y^{2u}+\beta_{s,t,u}Z^{2u})=0,
                          \end{equation}
Moreover, $\phi_{\sigma}\circ\,^{\sigma}\phi_{\sigma}=\operatorname{diag}(1,\lambda^{-1}\mu,\lambda\mu^{-1})\in\langle\operatorname{diag}(1,-1,1),\,\operatorname{diag}(1,1,-1)\rangle.$
Therefore, $\lambda=\pm\mu$ and $\phi_{\sigma}=[\epsilon X:Z:\pm Y]$ for some $d$-th root of unity $\epsilon$. In particular, $^{\sigma}\beta_{s,t,u}=\epsilon^{2s}\alpha_{s,t,u}$, and equation (\ref{geomcompletefamilyremaining2}) reduces to
\begin{equation*}
X^d+Y^d+Z^d+\sum_{\tiny{\begin{array}{c}
                           0\leq s,t,u\leq (d/2)-1\\
                             s+t+u=d/2
                          \end{array}}}\,X^{2s}(YZ)^{2t}(\alpha_{s,t,u}Y^{2u}+\epsilon^{-2s}\,^{\sigma}\alpha_{s,t,u}Z^{2u})=0.
                          \end{equation*}
In this case, we can take $\phi_{\sigma}=[\epsilon X:Z:Y]$, which satisfies $\phi_{\sigma}\circ\,^{\sigma}\phi_{\sigma}=1$. So, $\mathbb{R}$ is a field of definition for $\mathcal{S}$ by Theorem \ref{weilcocylcle}.

The remaining situations can be treated symmetrically.
\end{proof}

%%%%%%%%%%%%%%%%%%%%%%%%%%%%%%%%%%%%%%%%%%%%%%%%%%%%%%%%%%%%%%%%%%%%%%%%%%%%%%%


\begin{thebibliography}{10}

%\bibitem{arabcorgriff}
%E. Arbarello, M. Cornalba, P. A. Griffiths and J. Harris.
%\emph{Geometry of algebraic
%curves I}.
%Grundlehren \textbf{267}, Springer-Verlag, 1985.

\bibitem{AQC}
M. Artebani, S. Quispe and C. Reyes.
\emph{Automorphism groups of pseudoreal Riemann surfaces}.
J. Pure Appl. Algebra \textbf{221} (2017), 2383-2407, doi.org/10.1016/j.jpaa.2016.12.039.


\bibitem{SAUL1}
M. Artebani and S. Quispe.
\emph{Fields of moduli and fields of definition of odd signature curves}.
Arch. Math. \textbf{99} (2012), 333-344, doi:0.1007/s00013-012-0427-6.

\bibitem{BaBa1}
E. Badr, F. Bars.
\emph{On the locus of smooth plane curves with a fixed automorphism group}.
Mediterr. J. Math. \textbf{13} (2016), 3605-3627. doi:\,10.1007/s00009-016-0705-9.

\bibitem{BaBa3}
E. Badr,  F. Bars.
\emph{Automorphism groups of non-singular plane curves of degree 5}.
Commun. Algebra \textbf{44} (2016), 327-4340. doi:\,10.1080/00927872.2015.1087547.

\bibitem{BaBa2}E. Badr and F. Bars.
\emph{Non-singular plane curves with an element of ``large" order in its automorphism group }.
Int. J. Algebra Comput. 26 (2016), 399-434. doi:\,10.1142/S0218196716500168.

%\bibitem{BaBaEl1}
%E. Badr, F. Bars, and Lorenzo E.. \emph{On twists of smooth plane curves}, arXiv:1603.08711v1. Submitted

\bibitem{BaHidQu} E. Badr, R. A. Hidalgo, and S. Quispe. \emph{Riemann surfaces defined over the reals}. Arch. Math. \textbf{110} (2018), 351-362, doi:10.1007/s00013-017-1146-9.

\bibitem{Bars}
F. Bars.
\emph{On the automorphisms groups of genus 3 curves}.
Surveys in Math. and Math. Sciences, {\textbf{2}(2)(2012), 83-124.}

%%\bibitem{BaBaHu} Badr E. and Bars F.; \emph{On field of definiton
%%and field of moduli for non-singular plane curve with cyclic
%%automorphism group}, work in progress.
%
%\bibitem{Chang}Chang H. C., \emph{On plane algebraic curves}, Chinese J. Math. \textbf{6} (1978), no. 2, 185-189.
%
%\bibitem{Cha} Ch$\hat{a}$telet F., \emph{Variations sur un th\`eme de H. Poincar\'e}, Ann. Sci. Ec. Norm. Sup.
%61 (1944), 249ï¿"1¤70.
%
%%\bibitem{Cr} S. Crass, \emph{Solving the sextic by iteration:A
%%study in complex geometry and dinamics}. Experimental Mathematics
%%8(3)(1999), 209-240
%
%\bibitem{dolgachev2009finite}
%I. Dolgachev and V. Iskovskikh.
%\emph{Finite subgroups of the plane Cremona group}.
%Algebra, Arithmetic, and Geometry, Progress in Mathematics Volume \textbf{269} (2009), 443-548.

\bibitem{Dol}
I. Dolgachev.
\emph{Classical Algebraic Geometry: a modern view}.
Cambridge Univ. Press 2012, see also Private Lecture Notes in: http://www.math.lsa.umich.edu/$\sim$idolga/.

%%\bibitem{Elisa1}Fit\'e F., Lario, Joan C. , Abd\'o R., \emph{Reticles i classes de conjugaci�o de grups finits d�ordre petit}.
%
%\bibitem{GAP}  The GAP Group, GAP -Groups, Algorithms,  and Programming, Version 4.5.7, 2012. (http://www.gap-system.org)
%
%\bibitem{Graaf} W. A. de Graaf, Harrison M., Pilnikova J., and Schicho J.; \emph{A Lie algebra method for
%rational parametrization of Severi-Brauer surfaces}. arXiv:math/0501157v2 [math.AG]
%
%%\bibitem{Gi-SZ} P. Gille, T.Szamuely;
%%Central simple algebras and Galois cohomology, a book of Cambridge
%%Univ.Press, also available in the web.
%
%

%\bibitem{DeEm1}
%P. D`ebes and M. Emsalem.
%\emph{On fields of moduli of curves}.
%J. Algebra \textbf{211} (1999), no. 1, 42-56.

\bibitem{earle} C. J. Earle.
\emph{On the moduli of closed Riemann surfaces with symmetries}.
Advances in the Theory of Riemann Surfaces. Ann. Math. Studies \textbf{66} (1971), 119-130.
\bibitem{Ha}
T. Harui.
\emph{Automorphism groups of plane curves}.
arXiv: 1306.5842v2[math.AG] 7 Jun 2014.

\bibitem{He}
P. Henn.
\emph{Die Automorphismengruppen dar algebraischen Functionenkorper vom Geschlecht 3}. Inagural-dissertation, Heidelberg, 1976.

\bibitem{hidalgo1}
R. Hidalgo.
\emph{Non-hyperelliptic Riemann surfaces with real field of moduli but not definable over the reals}.
Arch. Math. \textbf{93} (2009), 219-224.

\bibitem{Hidalgo} R. Hidalgo.
 \emph{The fields of moduli of a Quadrangular Riemann surface}.
 Preprint (2010). http://arxiv.org/abs/1203.6317


\bibitem{howe}
E. W. Howe.
\emph{Plane quartics with jacobians isomorphic to a hyperelliptic jacobian}.
Proceedings of the American Mathematical Society 129 (2000), 1647-1657.


%\bibitem{Ha2}
%T. Harui, T. Kato, J. Komeda and A. Ohbuchi, \emph{Quotient curves of smooth plane curves with automorphisms}, Kodai Math. J.
%\textbf{33} (2010), 164-172.
%


%

\bibitem{Hug} B. Huggins. \emph{Fields of moduli and fields of definition of
curves}. PhD thesis, Berkeley (2005), see
http://arxiv.org/abs/math/0610247v1.
%
\bibitem{Hug2} B. Huggins. \emph{Fields of moduli of hyperelliptic curves}. Math. Res. Lett. \textbf{14} (2007), 249-262.
%\bibitem{Ja} J. Jahnel; \emph{The Brauer-Severi variety associated with a
%central simple algebra: a survey}. See the book in
%https://www.math.uni-bielefeld.de/lag/man/052.pdf
%
\bibitem{LeRi} R. Lercier and C. Ritzenthaler. Hyperelliptic curves and their
invariants: geometric, arithmetic and algorithmic aspects. J.
Algebra, 372:595–636, 2012.
\bibitem{Koizumi} S. Koizumi. \emph{Fields of moduli for polarized abelian varieties and for curves},
Nagoya Math. J. \textbf{48} (1972), 37-55.

\bibitem{LRS} R. Lercier, C. Ritzenthaler, and J. Sijsling. \emph{Explicit galois
obstruction and descent for hyperelliptic curves with tamely cyclic
reduced automorphism group}. Math. Comp, To appear.

\bibitem{Kont}
A. Kontogeorgis.
\emph{Field of moduli versus field of definition for cyclic covers of
the projective line}.
J. de Theorie des Nombres de Bordeaux \textbf{21} (2009) 679-692.

%\bibitem{LeRiRo} R. Lercier, C. Ritzenthaler, F. Rovetta, J. Sijsling. \emph{Parametrizing the moduli space of curves and applications to smooth plane quartics over finite fields}.
%(LMS Journal of Computation and Mathematics, Volume 17, Special
%Issue A (ANTS XI), LMS, London, pp. 128--147, 2014

\bibitem{Me} J.-F. Mestre.
\emph{Construction de courbes de genre 2 a partir de
leurs modules}. In Effective methods in algebraic geometry
(Castiglioncello, 1990) , volume 94 of Progr. Math. , pages 313–334.
Birkh\"auser Boston, Boston, MA, 1991.

\bibitem {Mit}H. Mitchell. \emph{Determination of the ordinary and modular ternary linear groups}. Trans.
Amer. Math. Soc. \textbf{12}, no. 2 (1911), 207-242.

\bibitem{shimura} G. Shimura. \emph{On the field of rationality for an abelian variety}. Nagoya Math. J.
\textbf{45} (1971), 167-178.

\bibitem{Silhol}
R. Silhol.
\emph{Moduli problems in real algebraic geometry.}
Real Algebraic Geometry (1972), 110--119.
Ed. M. Coste et al. (Springer-Verlag, Berlin).

\bibitem{Sing}
D. Singerman.
\emph{Symmetries and Pseudosymmetries of hyperelliptic surfaces}. Glasgow Math. \textbf{21}
(1980), 39–49.

\bibitem{We} A. Weil. \emph{The field of definition of a variety}. American J. of Math. vol. 78, nï¿"1¤7(1956),
509--524.
\end{thebibliography}
\end{document}